\documentclass[a4paper]{amsart}
\usepackage{amsmath,amsthm,mathrsfs,amssymb,mathtools,microtype,tikz-cd,hyperref}
\usepackage[T1]{fontenc}
\usepackage[capitalise]{cleveref}

\usepackage[foot]{amsaddr}
\makeatletter
\@namedef{subjclassname@2020}{\textup{2020} Mathematics Subject Classification}
\makeatother

\usepackage[left=3.5cm,right=3.5cm,top=3.25cm,bottom=3.25cm]{geometry}
\usepackage{setspace}
\theoremstyle{plain}
\newtheorem{thm}{Theorem}[section]
\newtheorem*{thm*}{Theorem}
\newtheorem{prop}[thm]{Proposition} 
\newtheorem{lem}[thm]{Lemma}
\theoremstyle{definition} 
\newtheorem{defn}[thm]{Definition} 
\newtheorem{exmp}[thm]{Example}
\theoremstyle{remark}
\newtheorem{rem}[thm]{Remark}

\DeclareMathOperator{\Der}{Der}
\DeclareMathOperator{\End}{End}
\DeclareMathOperator{\Hom}{Hom}
\DeclareMathOperator{\MC}{MC}
\DeclareMathOperator{\id}{id}
\newcommand{\op}{\mathrm{op}}
\newcommand{\Alg}{\mathsf{Alg}}
\newcommand{\DGA}{\mathsf{DGA}}
\newcommand{\CDG}{\mathsf{CDG}}
\newcommand{\DGMod}{\mathsf{DGMod}\textrm{-}}
\newcommand{\pcAlg}{\mathsf{pcAlg}}
\newcommand{\pcDGA}{\mathsf{pcDGA}}
\newcommand{\pcCDG}{\mathsf{pcCDG}}
\newcommand{\pcDGMod}{\mathsf{pcDGMod}\textrm{-}}
\newcommand*\lon{\nobreak \mskip6mu plus1mu \mathpunct{} \nonscript \mkern-\thinmuskip{:}\mskip2mu \relax}
\makeatletter
\newcommand{\invlim}{\mathop{\mathpalette\varlim@{\leftarrowfill@\scriptstyle}}\nmlimits@}
\newcommand{\dirlim}{\mathop{\mathpalette\varlim@{\rightarrowfill@\scriptstyle}}\nmlimits@}
\makeatother
\DeclareFontFamily{U}{mathx}{\hyphenchar\font45} 
\DeclareFontShape{U}{mathx}{m}{n}{<-> mathx10}{}
\DeclareSymbolFont{mathx}{U}{mathx}{m}{n}
\DeclareMathAccent{\widecheck}{0}{mathx}{"71} % get \widecheck
\DeclareMathAccent{\widebar}{0}{mathx}{"73} % get \widebar

\title{Koszul duality for compactly generated derived categories of second kind}

\author{Ai Guan}
\email{a.guan@lancaster.ac.uk}

\author{Andrey Lazarev}
\email{a.lazarev@lancaster.ac.uk}

\address{Department of Mathematics and Statistics, Lancaster University, Lancaster LA1 4YF, UK}

\thanks{This work was partially supported by the EPSRC grant EP/N015452/1}

\begin{document}

\begin{abstract}
For any dg algebra $A$ we construct a closed model category structure on dg $A$-modules
such that the corresponding homotopy category is compactly generated by dg $A$-modules that are finitely generated and free
over $A$ (disregarding the differential). We prove that this closed model category is Quillen equivalent to the category
of comodules over a certain, possibly nonconilpotent dg coalgebra, a so-called extended bar construction of $A$. 
This generalises and complements certain aspects of dg Koszul duality for associative algebras.
\end{abstract}
\maketitle

\section{Introduction}

Koszul duality is an algebraic phenomenon that goes back to Quillen's work \cite{Qui69} on rational homotopy theory; it later manifested itself in many different contexts: operads \cite{GiK94}, deformation theory \cite{Hin01}, representation theory \cite{BGS96} and numerous others. 

The modern understanding of Koszul duality for differential graded (dg) algebras and dg modules has been formulated in~\cite{Pos11}. 
According to this formulation there is an adjunction between the categories of augmented dg algebras and conilpotent dg coalgebras, given by bar and cobar constructions, which becomes a Quillen equivalence under certain model category structures. 
The conilpotent dg coalgebra associated to an augmented dg algebra by this equivalence is called its Koszul dual; similarly the augmented dg algebra associated to a conilpotent dg coalgebra is called its Koszul dual. 
There is also a Quillen equivalence between the corresponding categories of dg modules and dg comodules. 
A variant of this correspondence exists for non-augmented dg algebras and their modules.

A salient feature of this theory is that the closed model category structures on the Koszul dual side (both for coalgebras and their comodules) are of the ``second kind'': the weak equivalences are not created in the underlying chain complexes but are of a more subtle nature (so-called \emph{filtered quasi-isomorphisms}). 

The module-comodule Koszul duality is the easiest one to prove (though still quite nontrivial), essentially because of its linear character: this is a duality between stable closed model categories whose homotopy categories are triangulated.
There are two symmetric versions of it: 
\begin{enumerate}
\item the duality between modules over a dg algebra and dg comodules over its Koszul dual conilpotent dg coalgebra and
\item the duality between comodules over a conilpotent dg coalgebra and dg modules over its Koszul dual dg algebra.
\end{enumerate}
What happens if one drops the condition of conilpotency on the coalgebra side?
The closed model structure on the category of comodules does not depend on the conilpotency assumption, \cite[Theorem 8.2]{Pos11}.
Furthermore, Positselski proves (\cite[Theorem 6.7]{Pos11}) that there is a Koszul duality between dg comodules over a possibly nonconilpotent dg coalgebra and modules over its Koszul dual dg algebra.
However, this time \emph{both} closed model structures are of the second kind: the weak equivalences on dg modules are not merely quasi-isomorphisms.
If the coalgebra happens to be conilpotent, then the duality specialises to the ordinary one: the Koszul dual dg algebra becomes cofibrant and weak equivalences of dg modules over a cofibrant dg algebra are ordinary quasi-isomorphisms.

In the present paper we construct a complementary version of Positselski's non-conilpotent Koszul duality as a Quillen equivalence between closed model categories of dg modules over a dg algebra and comodules over its ``Koszul dual'' dg coalgebra.
The difference between our version and the standard one is two-fold: firstly, the weak equivalences between dg modules are of ``second kind'' (i.e.~they are not created in the category of underlying complexes) and secondly, our ``Koszul dual'' dg coalgebra is typically much bigger than the ordinary bar construction; in particular it is \emph{not} conilpotent in general.
This extended bar construction has been considered, e.g.~in a recent paper \cite{AJ13}. 

Perhaps the most interesting feature of this correspondence is an exotic  model structure of the second kind on dg modules over a dg algebra: in the case of an ordinary algebra (or, more generally, cohomologically non-positively graded dg algebra) this structure reduces to the usual one; however in general it is different.
There are many competing inequivalent notions of weak equivalence of the second kind for dg modules over a dg algebra (as opposed to dg comodules where there is only one such notion); some of them support closed model category structures, \cite[Proposition 1.3.6]{bec14}, \cite[Theorem 8.3]{Pos11}.
Our structure is generally different from those considered in the mentioned references and characterised by its compatibility with Koszul duality.
It is, necessarily, compactly generated (since such is the category of dg comodules over any dg coalgebra, to which it is Quillen equivalent). This model structure is relevant to the study of various triangulated categories of geometric origin: coherent sheaves on complex analytic manifolds, cohomologically constant sheaves on smooth manifolds, and $D$-modules on smooth algebraic varieties.
Its prototype is contained in the paper \cite{Blo10} where the notion of a cohesive module over a dg algebra is introduced, which is essentially the same as a cofibrant object in our closed model structure.

\subsection{Notation and conventions} 

Throughout the paper, $k$ denotes a field. 
All vector spaces will be over $k$ and differential graded (dg) vector spaces are further assumed to be cohomologically $\mathbb{Z}$-graded.
Given a graded vector space $V$, its suspension $\Sigma V$ is a graded vector space with $(\Sigma V)^i = V^{i+1}$ and its dual $V^*$ is a graded vector space with $(V^*)^i = (V^{-i})^*$. 
Unadorned tensor products and Homs are assumed to be over $k$.
The category of (graded) algebras is denoted $\Alg$, the category of dg algebras is denoted $\DGA$ and the category of augmented dg algebras is denoted $\DGA^*$; all of these are also implicitly assumed to be over $k$. 

A pseudocompact vector space is a projective limit of finite-dimensional vector spaces, equipped with the inverse limit topology. 
In particular, the $k$-linear dual $V^*$ of a discrete vector space $V$ is pseudocompact, and a finite-dimensional vector space is pseudocompact if and only if it is discrete.
Given a pseudocompact vector space $V$, its dual $V^*$ is defined to be its topological dual, hence $V \cong V^{**}$ is always true. 
Given pseudocompact vector spaces $V$ and $W$, the space of morphisms $\Hom(V,W)$ is assumed to mean the space of continuous linear maps, and the tensor product $V \otimes W$ is assumed to mean the completed tensor product.
If $V = \invlim_i V_i$ is pseudocompact and $W$ is discrete, then their tensor product is defined to be $V \otimes W = \invlim_i V_i \otimes W$; note that in general this is neither discrete nor pseudocompact.
The category of (graded) pseudocompact algebras is denoted $\pcAlg$, the category of pseudocompact dg algebras is denoted $\pcDGA$ and the category of augmented pseudocompact dg algebras is denoted $\pcDGA^*$. 
These categories are anti-equivalent to the categories of (coaugmented) (dg) coalgebras via taking linear and topological duals. 

We will generally work with \emph{right} modules over dg algebras and pseudocompact dg algebras, unless stated otherwise.
Given a dg algebra $A$, the category of dg $A$-modules is denoted $\DGMod A$.
Given a pseudocompact dg algebra $C$, a pseudocompact $C$-module is a pseudocompact vector space $V$ together with a continuous linear map $V \otimes C \to V$ satisfying the usual axioms of associativity and unitality.
The category of pseudocompact dg $C$-modules is denoted $\pcDGMod C$; this category is anti-equivalent to the category of dg $C^*$-comodules, again via taking duals.
Thus, all our results concerning pseudocompact dg modules can readily be translated into results about dg comodules if one wishes to do so.
%Given a dg (pseudocompact or not) algebra $A$, we will write $A^{\op}$ for its opposite dg algebra and for a category $\mathsf C$ the symbol ${\mathsf C}^{\op}$ will stand for the category with the same objects and reversed arrows.

\subsection{Acknowledgement}
The authors would like to thank Leonid Positselski for freely sharing his expert knowledge of the subject matter and Joe Chuang for many stimulating discussions.

\section{Extended bar construction}
\label{sec:barconst}

Given an algebra $A$, its \emph{pseudocompact completion} $\widecheck{A}$ is the projective limit of the inverse system of quotients by cofinite-dimensional ideals of $A$. 
Pseudocompact completion defines a functor from $\Alg \to \pcAlg$ that is left adjoint to the functor $\pcAlg \to \Alg$ forgetting the topology. 

\begin{defn}
Let $V$ be a pseudocompact graded vector space. 
If $V$ is finite-dimensional, its \emph{pseudocompact tensor algebra} $\widecheck{T}V$ is the pseudocompact completion of the tensor algebra $TV$.
For a general pseudocompact vector space $V = \invlim_i V_i$, its \emph{pseudocompact tensor algebra} is 
\[
\widecheck{T}V \coloneqq \invlim_i{\widecheck{T}V_i}.
\]
\end{defn}

\begin{prop}\label{pcfree}
Let\/ $V$ be a pseudocompact graded vector space. 
\begin{enumerate} 
\item\label{pcfree1} The pseudocompact tensor algebra\/ $\widecheck{T}V$ is the free pseudocompact algebra on\/ $V$, that is, for any pseudocompact algebra\/ $A$ there is a bijection
\[
\Hom_{\pcAlg}(\widecheck{T}V,A) \cong \Hom(V,A).
\]
\item For any pseudocompact\/ $\widecheck{T}V$-$\widecheck{T}V$-bimodule\/ $M$ there is a bijection
\[
\Der(\widecheck{T}V,M) \cong \Hom(V,M).
\]
\end{enumerate}
\end{prop}

\begin{proof}\hfill
\begin{enumerate} 
\item If $V$ is finite-dimensional, then $V$ is discrete and $\Hom(V,A) \cong \Hom_{\Alg}(TV,A)$, which equals $\Hom_{\pcAlg}(\widecheck{T}V,A)$ as pseudocompact completion is left adjoint to the forgetful functor.
More generally, writing $V = \invlim_i V_i$ and $A = \invlim_j A_j$ with $V_i$ and $A_j$ finite-dimensional, we have
\begin{align*}
\Hom_{\pcAlg}(\widecheck{T}V,A) 
&\cong \invlim_j \Hom_{\pcAlg}(\widecheck{T}V,A_j) \\
&\cong \invlim_j \dirlim_i \Hom_{\pcAlg}(\widecheck{T}V_i,A_j) \\
&\cong \invlim_j \dirlim_i \Hom(V_i,A_j)
\cong \Hom(V,A).
\end{align*}
Here, the second bijection holds as finite-dimensional algebras are cocompact in $\pcAlg$, that is, for any finite-dimensional algebra $A$, the functor $\Hom_{\pcAlg}(-,A)$ takes filtered limits to filtered colimits.
\item Recall the following construction, which allows us to turn questions about derivations into question about algebra homomorphisms.
Given a graded pseudocompact algebra $A$ and an $A$-$A$-bimodule $M$ consider the pseudocompact algebra $A \oplus M$ with multiplication $(a,m)\cdot(b,n)=(ab,an+mb)$, and let $p \colon A \oplus M \to A$ be the natural projection. 
Then there is a bijection 
$\Der(A,M) \cong \{f \in \Hom_{\pcAlg}(A,A\oplus M) : p \circ f = 1_A\}$.
Setting $A = \widecheck{T}V$ and using part (\ref{pcfree1}), we have 
\[
\Der(\widecheck{T}V,M) \cong \{f \in \Hom(V,\widecheck{T}V\oplus M) : p \circ f = 1_A\}
\cong \Hom(V,M). \qedhere 
\]
\end{enumerate}
\end{proof}

\begin{rem}
The pseudocompact algebra $\widecheck{T}V$ is the $k$-linear dual to the Sweedler cofree coalgebra on the discrete vector space $V^*$, \cite[Section 6.4]{Swe69}.
\end{rem}

\begin{prop}\label{pcTVresolution}
For any pseudocompact vector space\/ $V$, there is a bimodule resolution of\/ $\widecheck{T}V$ given by
\begin{center}
\begin{tikzcd}[sep=scriptsize]
0 \rar & \widecheck{T}V \otimes V \otimes \widecheck{T}V
\rar["d"] & \widecheck{T}V \otimes \widecheck{T}V \rar["m"] & \widecheck{T}V \rar & 0
\end{tikzcd}
\end{center}
where $m$ is multiplication and $d(1 \otimes v \otimes 1) = v \otimes 1 - 1 \otimes v$.
\end{prop}

\begin{proof}
We use the following well-known fact for algebras, that also holds in the pseudocompact setting. 
Let $(A,\mu)$ be a graded pseudocompact algebra. 
Then $\Omega(A) = \ker \mu$ is an $A$-$A$-bimodule and the map $\delta \colon A \to \Omega(A)$ given by $\delta(a) = a\otimes 1 - 1 \otimes a$ is a derivation. 
For any derivation $d \colon A \to M$ taking values in an $A$-$A$-bimodule $M$, there is a unique bimodule homomorphism $f \colon \Omega(A) \to M$ such that $d=f \circ \delta$; hence
\[
  \Der(A,M) \cong \Hom_{A\text{-}A}(\Omega(A),M).
\]

Now by \cref{pcfree}, $\Der(\widecheck{T}V,M) \cong \Hom(V,M) \cong \Hom_{\widecheck{T}V\text{-}\widecheck{T}V}(\widecheck{T}V\otimes V \otimes \widecheck{T}V,M)$, so $\Omega(\widecheck{T}V) \cong \widecheck{T}V \otimes V\otimes \widecheck{T}V$ as required.
\end{proof}

All our dg algebras are augmented, except in Section \ref{sec:curved}. 
The augmentation ideal of a dg algebra $A$ is denoted $\widebar{A}$.

\begin{defn}\label{dgabarcobar}
We define a pair of functors
\[
\Omega \colon (\pcDGA^*)^{\op} \rightleftarrows \DGA^* \lon \widecheck{B}
\]
as follows.
The \emph{cobar construction} associates to a pseudocompact dg algebra $C$ the dg algebra
\[
\Omega C \coloneqq T\Sigma^{-1}\widebar{C}^*
\]
with differential defined in the usual way.

The \emph{extended bar construction} associates to a dg algebra $A$ the pseudocompact dg algebra 
\[
\widecheck{B}A \coloneqq \widecheck{T}\Sigma^{-1}\widebar{A}^*.
\]
We define the differential on $\widecheck{B}A$ as follows: Let $d_1 \colon \Sigma^{-1}\widebar{A}^* \to \Sigma^{-1}\widebar{A}^*$ and $d_2 \colon \Sigma^{-1}\widebar{A}^* \to \Sigma^{-1}\widebar{A}^* \widehat{\otimes} \Sigma^{-1}\widebar{A}^*$ be induced by dualising the differential and multiplication on $A$ respectively.
For a pseudocompact vector space $V$, consider the \emph{semi-completed tensor algebra} 
$T'(V) = \bigoplus_{n \geq 1} V^{\widehat{\otimes} n}$, 
which has a topology that is neither pseudocompact nor discrete, and has the property $\Hom_{\Alg}(T'(V),B)\cong\Hom(V,B)$ for any pseudocompact algebra $B$ (see \cite[Lemma 4.5]{gua}). 
Then by \cref{pcfree}(1), the identity on $\widecheck{T}\Sigma^{-1}\widebar{A}^*$ induces a map $i \colon T'(\Sigma^{-1}\widebar{A}^*) \to \widecheck{T}(\Sigma^{-1}\widebar{A}^*)$, and we define the differential to be
\[
i \circ (d_1+d_2) \colon \Sigma^{-1}\widebar{A}^* \to T'(\Sigma^{-1}\widebar{A}^*) \to \widecheck{T}(\Sigma^{-1}\widebar{A}^*).
\]
\end{defn}

\subsection{The Maurer--Cartan functor and representability} 

Let $A$ be a dg algebra (possibly discrete, pseudocompact or otherwise). 
A \emph{Maurer--Cartan element} in $A$ is an element $x \in A$ of degree 1 such that $dx+x^2 = 0$. 
The set of all Maurer--Cartan elements in $A$ is denoted $\MC(A)$.
For any dg algebra $A$ and any pseudocompact dg algebra $C$, define $\MC(A,C) \coloneqq \MC(A \otimes C)$; this is functorial in both arguments.

\begin{prop}\label{prop:repres}
Let $A$ be an augmented dg algebra and $C$ be an augmented pseudocompact dg algebra. There are natural bijections
\[
\Hom_{\DGA^*}(\Omega C,A)\cong \MC(\widebar{A},\widebar{C}) \cong \Hom_{\pcDGA^*}(\widecheck{B}A,C).
\]
In particular, $\Omega$ is a left adjoint functor to $\widecheck{B}$.
\end{prop}

\begin{proof}
Forgetting the differential, any map of augmented pseudocompact algebras $f \colon \widecheck{B}A \to C$ is equivalent to a linear map $\Sigma^{-1}\widebar{A}^* \to \widebar{C}$ by \cref{pcfree}, which is equivalently a degree 1 element $x \in \widebar{A} \otimes \widebar{C}$.
The condition that $f$ commutes with differentials is then equivalent to condition that $x$ satisfies the Maurer--Cartan equation; this can be proven just like the corresponding statement for the non-extended bar construction, see for example~\cite[Proposition 2.2]{cl11}. The other bijection is proved similarly.
\end{proof}

\begin{rem}
An adjoint pair of functors $(\Omega, \operatorname{B^{ext}})$ between $\DGA^*$ and $\pcDGA^*$ was defined in \cite[Section 5.3]{AJ13} in a different way; it was also proved that that these functors represent the MC sets (called twisting cochains in op.~cit.) as in Proposition \ref{prop:repres}. 
It follows that these functors are (isomorphic to) the functors $\Omega$ and $\widecheck{B}$ defined above.
\end{rem}

\section{Koszul duality for modules}
\label{sec:kd}

\subsection{Maurer--Cartan twisting}

We begin this section by recalling the notion of Maurer--Cartan twistings of dg algebras and dg modules.

\begin{defn}\label{def:twisted}
Let $(A,d_A)$ be a dg algebra and $x \in \MC(A)$.
\begin{enumerate}
\item The \emph{twisted algebra of $A$ by $x$}, denoted $A^x = (A,d^x)$, is the dg algebra with the same underlying algebra as $A$ and differential $d^x(a) = d_A(a) + [x,a]$.
\item Let $(M,d_M)$ be a \emph{left} dg $A$-module.
The \emph{twisted module of $M$ by $x$}, denoted $M^{[x]} = (M,d^{[x]})$, is the left dg $A^x$-module with the same underlying module structure as $M$ and differential $d^{[x]} (m) = d(m) + xm$.
\end{enumerate}
\end{defn}

Furthermore, if $A$ and $B$ are dg algebras and $M$ is a dg $A$-$B$-bimodule, then for any $x \in \MC(A)$ the twisted module of $M$ by $x$ is a dg $A^x$-$B$-bimodule, that is, the right $B$-module action remains compatible with the new differential. 

\begin{defn}
A \emph{twisted $A$-module} is a dg $A$-module that is free as an $A$-module after forgetting the differential, that is, it is isomorphic as an $A$-module to $V \otimes A$ for some graded vector space $V$.
A \emph{finitely generated twisted $A$-module} is a twisted $A$-module $V \otimes A$ with $V$ finite-dimensional.
\end{defn}

Given any graded vector space $V$, the $A$-module $V \otimes A$ equipped with the differential $1 \otimes d_A$ is a twisted $A$-module. 
More generally, by considering $V \otimes A$ as a $(\End(V) \otimes A)$-$A$-bimodule, every twisted $A$-module is of the form $(V \otimes A, 1 \otimes d_A)^{[x]}$ for some $x \in \MC(\End{V} \otimes A)$, as noted in~\cite[Remark 3.2]{chl}.

\begin{defn}\label{functorsFG}
Let $A$ be an augmented dg algebra, and let $\widecheck{B}A$ be its extended bar construction.
Let $\xi \in \MC(A \otimes \widecheck{B}A)$ be the canonical Maurer--Cartan element corresponding to the counit $\Omega \widecheck{B}A \to A$ of the adjunction $\Omega \dashv \widecheck{B}$.
Define a pair of functors
\[
G \colon (\pcDGMod \widecheck{B}A)^{\op} \rightleftarrows \DGMod A \lon F
\]
as follows.
The functor $F$ associates to a dg $A$-module $M$ the pseudocompact dg $\widecheck{B}A$-module
\[
FM \coloneqq (M^* \otimes \widecheck{B}A)^{[\xi]}
\]
and the functor $G$ associates to a pseudocompact dg $\widecheck{B}A$-module $N$ the dg $A$-module
\[
GN \coloneqq (N^* \otimes A)^{[\xi]}.
\]
\end{defn}

The functors $F$ and $G$ are well-defined as $FM$ is a dg $(A \otimes \widecheck{B}A)^{\xi}$-$\widecheck{B}A$-bimodule and $GN$ is a dg $(\widecheck{B}A \otimes A)^{\xi}$-$A$-bimodule; the left $(A \otimes \widecheck{B}A)^{\xi}$-module structure on $FM$ is disregarded as similarly with $GN$.
It is a standard fact that $G$ is left adjoint to $F$; more generally this is true replacing $\widecheck{B}A$ with any pseudocompact dg algebra $C$ and $\xi$ with any Maurer--Cartan element in $A \otimes C$, see for example~\cite[Section 6.2]{Pos11}.

\begin{rem}
In the standard formulation of Koszul duality, the functors are defined as follows: the bar construction of a dg algebra $A$ is instead defined to be $BA = \widehat{T}\Sigma^{-1}\widebar{A}^*$, a \emph{local} or \emph{pronilpotent} pseudocompact dg algebra (or dually, a conilpotent dg coalgebra).
Given a dg $A$-module $M$, the corresponding $BA$-module is defined as $(M^* \otimes BA)^{[\xi]}$ where $\xi \in \MC(A \otimes BA)$ is the canonical Maurer--Cartan element corresponding to the counit $\Omega BA \to A$ of the Koszul duality adjunction for algebras.
Conversely, given a $BA$-module $N$, the corresponding $A$-module  is defined as $(N^* \otimes A)^{[\xi]}$. 
\end{rem}

\subsection{Model category structure on $\DGMod A$}

We now define model category structures on $\DGMod A$ and $\pcDGMod \widecheck{B}A$ making the adjunction $G \dashv F$ a Quillen pair.
In~\cite{Pos11} Positselski constructs a model category structure of the ``second kind'' on the category of dg comodules over an arbitrary (not necessarily conilpotent) dg coalgebra; this will be the model category structure on $\pcDGMod \widecheck{B}A$.
We begin by recalling this result.

\begin{defn}
Let $C$ be a dg coalgebra.
A dg $C$-comodule is \emph{coacyclic} if it is in the minimal triangulated subcategory of the homotopy category of dg $C$-comodules containing the total $C$-comodules of exact triples of dg $C$-comodules and closed under infinite direct sums.
\end{defn}

\begin{thm}\label{thm:modelcat} \cite[Theorem 8.2]{Pos11}
Let\/ $C$ be a dg coalgebra.
There exists a model category structure on the category of dg\/ $C$-comodules, where
\begin{enumerate}
\item a morphism $f \colon M \to N$ is a \emph{weak equivalence} if its cone is a coacyclic dg\/ $C$-comodule;
\item a morphism is a \emph{cofibration} if it is injective;
\item a morphism is a \emph{fibration} if it is surjective with a fibrant kernel.
\end{enumerate}
\end{thm}

Furthermore, this model category structure is cofibrantly generated, where generating cofibrations are injective maps between finite-dimensional comodules.

\begin{thm}\label{modcat}
Let\/ $A$ be an augmented dg algebra.
There is a cofibrantly generated model category structure on\/ $\DGMod A$, where 
\begin{enumerate}
\item a morphism $f \colon M \to N$ is a \emph{weak equivalence} if it induces a quasi-isomorphism
\[
\Hom_A((V \otimes A)^{[x]}, M) \to \Hom_A((V \otimes A)^{[x]}, N)
\]
for any finitely generated twisted $A$-module $(V \otimes A)^{[x]}$;
\item a morphism is a \emph{fibration} if it is surjective;
\item a morphism is a \emph{cofibration} if it has the left lifting property with respect to acyclic fibrations.
\end{enumerate}
With this model structure, the adjunction\/ $G \dashv F$ is a Quillen pair.
\end{thm}

To prove \cref{modcat}, we will apply the following version of the transfer principle, which appears in \cite[Sections 2.5--2.6]{bm03}.

\begin{thm*}[Transfer principle]
Let\/ $\mathsf{M}$ be a model category cofibrantly generated by the sets\/ $\mathcal{I}$ and\/ $\mathcal{J}$ of generating cofibrations and generating acyclic cofibrations respectively.
Let\/ $\mathsf{C}$ be a category with finite limits and small colimits.
Let
\[
L \colon \mathsf{M} \rightleftarrows \mathsf{C} \lon R
\]
be a pair of adjoint functors. 
Define a map $f$ in\/ $\mathsf{C}$ to be a weak equivalence \textup(respectively fibration\textup) if $R(f)$ is a weak equivalence \textup(respectively fibration\textup).
These two classes determine a model category structure on\/ $\mathsf{C}$ cofibrantly generated by $L(\mathcal{I})$ and $L(\mathcal{J})$ provided that:
\begin{enumerate}
\item\label{transfer-small} The functor $L$ preserves small objects;
\item\label{transfer-functorial} $\mathsf{C}$ has a functorial fibrant replacement and a functorial path object for fibrant objects.
\end{enumerate}
Furthermore, with this model structure on\/ $\mathsf{C}$, the adjunction $L \dashv R$ becomes a Quillen pair.
\end{thm*}

We first check that the weak equivalences and fibrations, obtained by transferring the model structure on $\pcDGMod \widecheck{B}A$ along the adjunction $G \dashv F$, admit the characterisations in \cref{modcat}. 
In fact, both the functors $F$ and $G$ preserve weak equivalences between all objects.

\begin{lem}\label{wepreserved}\hfill
\begin{enumerate}
\item A morphism $g$ of dg $A$-modules is a weak equivalence if and only if $F(g)$ is a weak equivalence.
\item A morphism $f$ of pseudocompact $\widecheck{B}A$-modules is a weak equivalence if and only if $G(f)$ is a weak equivalence.
\end{enumerate} 
\end{lem}

\begin{proof}
For (1), let $g \colon M \to N$ be a map of dg $A$-modules.
By definition $F(g) \colon FM \to FN$ is a weak equivalence if and only if it induces a quasi-isomorphism
\[
\Hom_{\widecheck{B}A}(FM,V) \to \Hom_{\widecheck{B}A}(FN,V)
\]
for any finite-dimensional dg $\widecheck{B}A$-module $V$.
Equivalently, this says that the  dg $A$-modules $M \otimes V$ and $N \otimes V$ (with possibly twisted diffferentials) are quasi-isomorphic for any finite-dimensional $V$, that is, $g$ is a weak equivalence. 

For (2), it suffices to show that $G$ takes exact triples of $\widecheck{B}A$-modules to weakly trivial $A$-modules. 
Let $N_1 \to N_2 \to N_3$ be an exact triple of $\widecheck{B}A$-modules and $N$ be its total complex. 
Then $GN$ is the total complex of the complex $G(N_3) \to G(N_2) \to G(N_1)$, which is a bicomplex with three vertical columns and the all horizontal rows exact. 

Now let $M=A\otimes V$ be a finitely generated twisted $A$-module. 
Applying $\Hom_A(M,-)$ to the above bicomplex gives $\Hom(V,G(N_3)) \to \Hom(V,G(N_2)) \to \Hom(V,G(N_1))$. 
Since exactness of the rows is preserved, $GN$ is indeed weakly trivial.
\end{proof}

\begin{lem}\label{fibrations}
A morphism $g$ of dg $A$-modules is a fibration if and only if $F(g)$ is a fibration.
\end{lem}

\begin{proof}
Let $g \colon M \to N$ be a fibration in dg $A$-modules, so $M \cong N \oplus V$ for some graded vector space $V$.
Then $F(g) \colon FN \to FM$ is a cofibration in $\pcDGMod \widecheck{B}A$ if and only if it is injective with cofibrant cokernel.
But indeed, $F(g) \colon (N^* \otimes \widecheck{B}A)^{[\xi]} \to (M^* \otimes \widecheck{B}A)^{[\xi]}$ has cokernel $(V^* \otimes \widecheck{B}A)^{[\xi]}$, which is cofibrant.
\end{proof}

\begin{proof}[Proof of \cref{modcat}] 
By \cref{wepreserved} and \cref{fibrations}, it suffices to check conditions (\ref{transfer-small}) and (\ref{transfer-functorial}) in the transfer theorem.
Condition (\ref{transfer-small}) holds as $G$ preserves small objects, and every object is fibrant so the first part of (\ref{transfer-functorial}) trivially holds. 
Hence it only remains to prove that functorial path objects exist for any $A$-module.
Let $I$ be the standard interval object for dg vector spaces, that is, $I = k \oplus \Sigma^{-1}k \oplus k$ with differential $d(a,b,c) = (da, -db+a-c, dc)$.
Then for any $A$-module $M$, there is a factorisation
\[
M \xrightarrow{e} M \otimes I \xrightarrow{(p_1,p_2)} M \oplus M
\]
where 
$e(a) = (a,0,a)$ and $p_1(a,b,c)=a$, $p_2(a,b,c)=c$. 
Clearly $(p_1, p_2)$ is a fibration by \cref{fibrations}. 
Since $I$ is acyclic, we have a quasi-isomorphism
\[
(M \otimes V^*)^{[x]} \to (M \otimes V^*)^{[x]} \otimes I \cong (M \otimes I \otimes V^*)^{[x]}
\]
for any finitely generated twisted $A$-module $(V \otimes A)^{[x]}$, so $e$ is a weak equivalence.
Thus $M \otimes I$ is a functorial path object for $M$.
\end{proof}

We now show that the adjoint pair $(F,G)$ is a Quillen equivalence.

\begin{thm}\label{Qequiv}
Let $A$ be an augmented dg algebra and $\widecheck{B}A$ be its extended bar construction.
\begin{enumerate}
\item For any pseudocompact $\widecheck{B}A$-module $N$, the unit $FGN \to N$ of the adjunction is a weak equivalence of pseudocompact $\widecheck{B}A$-modules. 
\item For any dg $A$-module $M$, the counit $GFM \to M$ of the adjunction is a weak equivalence
of $A$-modules.
\end{enumerate}
Thus, the Quillen adjunction $G \dashv F$ is a Quillen anti-equivalence between dg $A$-modules and pseudocompact $\widecheck{B}A$-modules.
\end{thm}

\begin{proof}
  For part (1), given a $\widecheck{B}A$-module $N$, the composition $FG(N)$ is the two-term resolution of $N$ from \cref{pcTVresolution}, so is weakly equivalent to $N$.
  Part (2) now follows immediately from part (1) as, by \cref{wepreserved}, the right Quillen functor $F$ reflects weak equivalences. 
  This is a standard fact; see for example \cite[Corollary 1.3.16]{hov99}.
\end{proof}

\begin{rem}
Note that the homotopy category of the constructed closed model category on dg $A$-modules is a compactly generated triangulated category (being anti-equivalent to the category of pseudocompact dg modules over a $\widecheck{B}A$) with compact (small) objects being dg modules that are homotopy equivalent to retracts of finitely generated twisted $A$-modules.
We will denote this homotopy category by $D^{\text{II}}_c(A)$.
\end{rem}

\begin{exmp}\label{smallexample}
Consider the dg algebra $A = k[x]/x^2$ with zero differential and $x$ in degree~1.
We have $\widecheck{BA} \cong \widecheck{k[x]}$.
If $k$ is algebraically closed then the pseudocompact completion $\widecheck{k[x]}$ of $k[x]$ is the product of completions of $k[x]$ at every maximal ideal of $k[x]$, the latter correspond precisely to elements of $k$.
In other words,
\[
\widecheck{BA} \cong \widecheck{k[x]} \cong \prod_{\alpha \in k} (k[[x]])_{\alpha}
\]
(this result, in a more general form, is given in \cite[Example 1.13]{gg99}).
The derived category $D^{\text{II}}_c(A)$ of $A$ of second kind is anti-equivalent to the derived category (of second kind) of pseudocompact modules over $\prod_{\alpha \in k} (k[[x]])_{\alpha}$ and thus, is drastically different from the ordinary derived category of $A$.
Note that $\MC(A) = \{ ax : a \in k\}$; then the twisted $A$-modules $A^{\xi}$ for $\xi\in\MC(A)$ are pairwise weakly inequivalent and form a set of compact generators for $D^{\text{II}}(A)$; it is easy to see that it is not possible to choose a single compact generator.
\end{exmp}
\begin{exmp}
The derived category of second kind $D^{\text{II}}_c$ arises in a number of situations of a geometric origin:\begin{itemize}
\item Let $M$ be a smooth manifold and $\mathscr{A}^*(M)$ be its smooth  de Rham algebra; here the ground field $k$ is $\mathbb{R}$, the real numbers. The choice of a point in $M$ makes $\mathscr{A}^*(M)$ into an augmented dg algebra. A compact object in $D_c^{\text {II}}(\mathscr{A}^*(M))$ is a cohesive $\mathscr{A}^*(M)$-module of \cite{Blo10} and the subcategory of compact objects is equivalent to the triangulated category of perfect cohomologically locally constant complexes of sheaves on $M$ by \cite[Theorem 8.1]{chl}.
\item Let $M$ be a smooth affine algebraic variety over a field $k$ of characteristic zero having a base point $\operatorname{Spec}(k)\to M$ and $\mathscr{A}^*(M)$ be its \emph{algebraic} de Rham algebra. Then twisted modules over $\mathscr{A}^*(M)$ correspond to $D$-modules, i.e. modules over the ring of differential operators on $M$ by \cite[Theorem B.2]{Pos11} while compact objects in $D_c^{\text {II}}(\mathscr{A}^*(M))$ correspond essentially to \emph{coherent} $D$-modules.
\item Let $M$ be a compact complex manifold and $\mathscr{A}^{0*}(M)$ be the Dolbeault algebra of $M$ that can be viewed as augmented by a choice of a base point in $M$. Again, a compact object  $D_c^{\text {II}}(\mathscr{A}^*(M))$ is a cohesive $\mathscr{A}^{0*}(M)$-module and the subcategory of compact objects is equivalent to the derived category of sheaves on $M$ with coherent cohomology, \cite[Theorem 4.1.3]{Blo10} or \cite[Theorem 8.3]{chl}  
\end{itemize}
\end{exmp}

\subsection{Comparison with other weak equivalences in $\DGMod A$}

Here, we compare the notion of weak equivalences in our model structure on $\DGMod A$ with other notions of a weak equivalence from the literature.

Firstly, we can consider the standard model structure on $\DGMod A$ where weak equivalences are quasi-isomorphisms and fibrations are surjections.
It is clear that any weak equivalence in our model structure is a quasi-isomorphism, by considering $A$-modules trivially twisted by the Maurer--Cartan element $x=0$.
It follows that $D^{\text{II}}_c(A)$ contains the ordinary derived category of $A$ as a full subcategory.
If $A$ is concentrated in nonpositive degrees (e.g.~it is an ordinary algebra), or $\bar{A}$ is concentrated in degrees $>1$ (e.g.~cohomology algebras of simply-connected topological spaces) then by the degree considerations, $\widecheck{B}A \cong \widehat{T}\Sigma^{-1}\widebar{A}^*$, the usual bar construction of $A$ from which it follows that our closed model structure on $A$-modules is the ordinary one (i.e.~of the first kind).
Another situation where we obtain the ordinary closed model category of the first kind is when the dg algebra $A$ is cofibrant.
However, for general $A$, even with a vanishing differential, we get a different result, cf.~\cref{smallexample}.

In \cite{Pos11}, the coderived category and contraderived category of a dg algebra $A$ are defined, which are obtained by localising at coacyclic dg $A$-modules and contraacyclic dg $A$-modules respectively.
These categories are different, in general, from the ordinary derived category of the first kind, even for ungraded algebras, see e.g.~\cite[Example 3.3]{Pos11} and thus, also from $D^{\text{II}}_c(A)$. 

It was observed in \cite{Pos11, KLN10} that the category $D^{\text{II}}_c(A)$ is contained in both the coderived and contraderived category of $A$.
It is, therefore, the derived category of $A$ of the second kind that is closest to the ordinary derived category of $A$.
If $A$ is right Noetherian and has finite right homological dimension then $D^{\text{II}}_c(A)$ coincides with both coderived and contraderived category of $A$ by \cite[Question 3.8]{Pos11}.
Another situation when this happens is when $A$ is the cobar construction of a (possibly nonconilpotent) dg coalgebra $B$ since in this case the co/contraderived category of $A$ is equivalent to the coderived category of $B$ and is, therefore, compactly generated.
Related questions are considered in the recent paper \cite{Pos17}.

\section{Curved Koszul duality for modules}\label{sec:curved}
\label{sec:curvedkd}

In this section, we consider generalisations of the previous results in the cases where the underlying dg algebra is curved or non-augmented. 
First we need to develop the extended bar-cobar formalism in the curved, non-augmented context.

A \emph{curved dg algebra} is a graded algebra $A$ with a degree one derivation $d \colon A \to A$, such that for any $a \in A$, $d^2(a)=[h,a]$ for some $h \in A^2$ satisfying $d(h)=0$. 
The linear map $d$ is usually called the \emph{differential} of $A$, despite not being square zero, and $h$ is called the \emph{curvature} of $A$. 

A morphism of curved algebras
$(A,d_A, h_A)\to (B,d_B, h_B)$ is a pair $(f,b)$ consisting of a morphism of graded algebras $f \colon A \to B$
and an element $b\in B^1$ satisfying the equations: 
\begin{align*}
&f(d_A(x)) = d_B(f(x)) + [b, f(x)],\\
&f(h_A) = h_B + d_B(b) + b^2,
\end{align*} 
for all $x\in A$; if $b=0$ then the corresponding morphism $A\to B$ is called \emph{strict}. 
The category of curved dg algebras is denoted $\CDG$ and the category of pseudocompact curved dg algebras is denoted $\pcCDG$; additionally we assume that our (pseudocompact or not) curved dg algebras have nonzero units. 
A \emph{Maurer--Cartan element} in a curved dg algebra $A$ is an element $x \in A$ of degree 1 such that $h+dx+x^2 = 0$. Given two curved dg algebras $(A, d_A, h_A)$ and $(B, d_B, h_B)$ their tensor product $A\otimes B$ is likewise a curved dg algebra with $d_{A\otimes B}:=d_A\otimes 1+1\otimes d_B$ and $h_{A\otimes B}:=h_A\otimes 1+1\otimes h_B$.

Given a curved dg algebra $(A,d_A, h_A)$ and an element $b\in A^1$ (not necessarily Maurer--Cartan) we can define the \emph{twisting} of $A$ by $b$ as a curved dg algebra $A^b$ with the same underlying vector space as $A$, twisted differential $d^b(x):=d_A(x)+[b,x]$ for $x\in A$ and curvature $h^b:=h_A+ d_A(b)+b^2$. 
Then $(\id, b)$ determines a (curved) isomorphism $A^b \to A$.

If $A$ is a curved dg algebra, then a \emph{dg $A$-module} is a graded (right) $A$-module $M$ with a degree one derivation $d_M \colon M \to M$ such that $d_M$ is compatible with the differential $d$ on $A$, and for any $m \in M$, $d_M^2(m)=mh$; one can similarly define left dg $A$-modules. 
If $M$ is a left dg $A$-module and $x\in A^1$, then there is a left dg $A^x$-module $M^{[x]}$ defined as in the uncurved case, cf.~\cref{def:twisted}.
Given a curved dg algebra $A$ and a pseudocompact curved dg algebra $C$, we denote the categories of dg $A$-modules and pseudocompact $C$-modules by $\DGMod A$ and $\pcDGMod C$, just as before. 

We now describe how to modify the bar and cobar constructions from \cref{dgabarcobar} in the general non-augmented and curved case. 
Let $A$ be a unital curved dg algebra with differential $d$ and curvature $h$. 
Since $1\neq 0$ in $A$ we can choose a homogeneous $k$-linear retraction $\epsilon \colon A \to k$, to be regarded as a ``fake augmentation''. 
It allows one to identify the dg vector space $\widebar{A}:=A/k$ with a subspace (possibly not dg) of $A$ so that $A\cong k\oplus \widebar{A}$. 
The multiplication $m \colon A\otimes A\to A$ restricted to $\widebar{A}$ has two components 
$m^{\epsilon}_{\widebar{A}} \colon \widebar{A}\otimes\widebar{A} \to \widebar{A}$ and 
$m_k^{\epsilon} \colon \widebar{A}\otimes\widebar{A} \to k$.
We will denote the corresponding components of the differential $d$ and curvature $h$ by $d_{\widebar A}^{\epsilon}, d_k^{\epsilon}$ and $h_{\widebar A}^{\epsilon}$, respectively; note that the component $h_k^{\epsilon}$ vanishes for degree reasons. 
Explicitly, for all $\widebar{a}, \widebar{b}\in\widebar{A}\subset A$,
\begin{equation}\label{eq:fakeaug}
\begin{split}
&m^{\epsilon}_{\widebar{A}}(\widebar{a},\widebar{b})=\widebar{a}\widebar{b}-\epsilon(\widebar{a}\widebar{b}),\ 
m^{\epsilon}_k(\widebar{a},\widebar{b})=\epsilon(\widebar{a}\widebar{b});\\ 
&d_{\widebar{A}}^{\epsilon}(\widebar{a})=d(\widebar{a})-\epsilon(d(\widebar{a})),\ 
d^{\epsilon}_k(\widebar{a})=\epsilon(d(\widebar{a}));\\ 
&h_{\widebar{A}}^{\epsilon}=h-\epsilon(h)=h.
\end{split}
\end{equation}
To alleviate notation, we will suppress the superscript $\epsilon$ at $m_{\widebar{A}}$, $m_k$ etc.~where it does not lead to confusion. 

Consider the graded algebra $T^{\prime}\Sigma^{-1}A^*$, the non-reduced semi-completed bar construction of $A$.  
Choose a basis $\{t^i : i\in I\}$ in $\widebar{A}$ where $I$ is some indexing set and let $\{\tau, t_i : i\in I\}$ be the basis in $\Sigma^{-1}A^*$ dual to the basis $\{1,t_i : i\in I\}$ in $A$. 
We will write $\partial_{t_i}$ for the derivation of $T^{\prime}\Sigma^{-1}A^*$ having value $1$ on $t_i$ and zero on other basis elements of $\Sigma^{-1}A^*$ and similarly for $\partial_\tau$. 
Then define the differential on $T^{\prime}\Sigma^{-1}A^*$ as the following derivation:
\[
\xi:=\sum_{i\in I}([\tau, t_i]+f_i(\mathbf t))\partial_{t_i}
+(g(\mathbf{t})+\tau^2)\partial_\tau+\sum_{i\in I}a_i\partial_{t_i}
\]
where $f_i(\mathbf t)$, $g(\mathbf t)$ stand for sums of linear and quadratic monomials in $t$ (so these elements of $T^{\prime}\Sigma^{-1}A^*$ do not depend on $\tau$). 
Here 
the term $\sum_{i\in I}f_i(\mathbf t)\partial_{t_i}$ corresponds to the ``multiplication and differential'' $m_{\widebar A}$ and $d_{\widebar A}$, 
the term $\sum_{i\in I}a_i\partial_{t_i}$ reflects the curvature $h_{\widebar A}$, 
the term $g(\mathbf{t})\partial_\tau$ corresponds to $m_k$ and $d_k$, 
and the term $(\sum_{i\in I}[\tau, t_i]+\tau^2)\partial_\tau$ reflects the multiplication with the unit in $A$. 
Let $\xi_1:=\sum_{i\in I}f_i(\mathbf t)\partial_{t_i}
+\sum_{i\in I}a_i\partial_{t_i}$ and $\xi_2:=\sum_{i\in I}[\tau, t_i]\partial_{t_i}
+(g(\mathbf{t})+\tau^2)\partial_\tau$; then $\xi=\xi_1+\xi_2$.

The \emph{reduced} semi-complete bar construction $B^\prime_{\epsilon} A$ of $A$ is a subalgebra in $T^{\prime}\Sigma^{-1}A^*$ spanned by sums of monomials which do not depend on $\tau$ (so only depend on $t_i$, $i\in I$). 
Thus, the underlying graded algebra of $B^\prime_{\epsilon} A$ is isomorphic to $T^\prime\Sigma^{-1}\widebar{A}^*$. 
The differential on $B^\prime_{\epsilon}A$ is $\xi_1$. 
Note that $\xi^2=0$ but $\xi_1^2=0$ only when $\epsilon \colon A \to k$ is a dg algebra map; in this case $g(\mathbf t)=0$.   
However $(B^\prime_{\epsilon} A, \xi_1)$ is a \emph{curved} dg algebra, more precisely the following result holds.

\begin{lem}\label{lem:semicomplete}
Let $A$ be a curved dg algebra. 
Then: 
\begin{enumerate} 
\item The reduced semi-complete bar construction $B^\prime_{\epsilon}A$ endowed with the differential $\xi_1$ defined above, is a curved dg algebra with curvature $-g(-\mathbf t)$, an element of $T'\Sigma^{-1}\widebar{A}^*$ obtained from $-g(\mathbf t)$ by replacing every indeterminate $t_i$ with $-t_i$.
\item\label{lem:semicomplete2} The curved dg algebra $B'_{\epsilon}A$ is independent, up to a natural isomorphism, of the choice of a basis in $\widebar{A}$. 
Furthermore, for different choices of retractions ${A \to k}$, the corresponding reduced semi-complete bar constructions are isomorphic as curved dg algebras. 
More precisely, denote by\/ $b_{\epsilon-\epsilon'}$ the element in $B^\prime A\cong T^\prime\Sigma^{-1}\widebar{A}^*$ corresponding to the linear map $\epsilon-\epsilon' \colon A\to k$; then the curved map $(\id, b_{\epsilon-\epsilon^\prime})$ determines a curved isomorphism $B'_{\epsilon} A \to B'_{\epsilon'}A$.
\item The correspondence $A\to B'_{\epsilon}A$ determines a contravariant functor from the category $\CDG$ to the category of topological curved dg algebras.
\end{enumerate}
\end{lem}

\begin{proof}
Taking into account that $0=\xi^2=\xi_1^2+[\xi_1,\xi_2]+\xi_2^2$ we have for $k\in I$,
\[
\xi_1^2(t_k)=-[\xi_1,\xi_2](t_k)-\xi_2^2(t_k).
\]
Furthermore, a straightforward calculation shows that $[\xi_1,\xi_2](t_k)$ has no terms depending on $t_i$, $i\in I$ whereas
the only term of $\xi_2^2(t_k)$ depending on $t_i$, $i\in I$ has the form $g(\mathbf t)\partial_\tau([t_k,\tau])=(-1)^{|t_k|}[t_k,g(\mathbf t)]$. It follows that
\[
\xi_1^2(t_k)=-(-1)^{|t_k|}[g(\mathbf t),t_k]
\]
as required.

Next, the statement about the independence of $B_{\epsilon}^\prime(A)$ on a basis in $\widebar{A}$ is obvious. Let $\epsilon' \colon A\to k$ be another fake augmentation; then formulas (\ref{eq:fakeaug}) show that $h$ is unchanged whereas 
$m_{\widebar{A}}^{\epsilon^\prime}(\widebar{a},\widebar{b})=m_{\widebar{A}}^{\epsilon}(\widebar{a},\widebar{b})+(\epsilon-\epsilon^\prime)(\widebar{a}\widebar{b})$,
and similarly for the differential.
This implies that $B^{\prime}_{\epsilon^{\prime}}A$ is obtained from $B^{\prime}_{\epsilon}A$ by twisting with the element $\epsilon-\epsilon^\prime\in B^{\prime}_{\epsilon}A$, which is equivalent to the stated claim.

To see that the construction $A\to B^\prime_\epsilon A$ is functorial, we will view an object in $\CDG$ as a curved dg algebra $A$ \emph{with a choice of a retraction} $A\to k$, however morphisms need not respect the retraction; this is clearly the same as (or, more accurately, equivalent to) the category $\CDG$. 
Any map $A\to B$ in $\CDG$ can canonically be factorised in $\CDG$ as $A\to A\to B$ with the first map being a change of retraction in $A$ followed by a map preserving retractions. 
The construction $B^\prime_{\epsilon} A$ is clearly functorial with respect to retraction-preserving maps and a change of retraction is also functorial by part (\ref{lem:semicomplete2}). 
\end{proof}

This allows us to define the extended bar construction of a curved dg algebra in the same way as it was done in the uncurved case; from now on we will suppress the subscript $\epsilon$ and write $B^\prime_{\epsilon}A$ for the semi-complete bar construction of $A$; by \cref{lem:semicomplete} this specifies a curved pseudocompact dg algebra up to a natural isomorphism.

\begin{defn}Let $A$ be a curved dg algebra with a retraction $\epsilon \colon A \to k$.
The \emph{extended bar construction} of $A$ is the graded pseudocompact algebra 
\[
\widecheck{B}A \coloneqq \widecheck{T}\Sigma^{-1}\widebar{A}^*.
\]
Then by \cref{pcfree}(1), the identity on $\widecheck{T}\Sigma^{-1}\widebar{A}^*$ induces a map $i \colon B'A \cong T'(\Sigma^{-1}\widebar{A}^*) \to \widecheck{B}A\cong\widecheck{T}(\Sigma^{-1}\widebar{A}^*)$, and we define the differential $d_{\widecheck{B}A}$ on $\widecheck{B}A$ to be
\[
d_{\widecheck{B}A}:=i \circ \xi_1 \colon \Sigma^{-1}\widebar{A}^* \to T'(\Sigma^{-1}\widebar{A}^*) \to \widecheck{T}(\Sigma^{-1}\widebar{A}^*).
\] The curvature of $\widecheck{B}A$ is the image of the curvature in $B^\prime A$ under the map $i \colon B'A \to \widecheck{B}A$. 
This gives $\widecheck{B}A$ the structure of a curved pseudocompact dg algebra.
\end{defn}

\begin{rem}
It follows from Lemma \ref{lem:semicomplete} that the correspondence $A\mapsto \widecheck{B}A$ is a functor $\CDG\to \pcCDG^{\op}$. 
A version of the definition above with $\widehat{T}\Sigma^{-1}\widebar{A}^*$ (the \emph{local} pseudocompact bar construction of a curved non-augmented algebra) in place of $\widecheck{T}\Sigma^{-1}\widebar{A}^*$ is found in \cite[Section 6.1]{Pos11}, albeit formulated in the language of coalgebras. 
However Positselski's local bar construction is \emph{not} functorial with respect to non-strict maps in $\CDG$ since maps between pseudocompact algebras of the form $\widehat{T}\Sigma^{-1}\widebar{A}^*$ must preserve their maximal ideals whereas this is not true for pseudocompact algebras of the form $\widecheck{T}\Sigma^{-1}\widebar{A}^*$ (which can have many maximal ideals). 
\end{rem}

Now recall that given a pseudocompact curved dg algebra $C$ there is defined a curved dg algebra
\[
\Omega C \coloneqq T\Sigma^{-1}\widebar{C}^*
\]
with $\widebar{C}:=C/k$, cf.~\cite[Section 6.1]{Pos11}.
Note that the definition of $\Omega$ can be given along the lines of the definition of $\widecheck{B}$, only simpler since there is no analogue, or need, for an intermediate step involving the semi-complete bar construction. 
Then we have the following result.

\begin{prop}
The correspondence $C \mapsto \Omega(C)$ determines a functor\/ $\pcCDG^{\op}\to\CDG$. 
This functor is left adjoint to $\widecheck{B} \colon \CDG\to \pcCDG^{\op}$. 
\end{prop}

\begin{proof}
The functoriality of $\Omega C$ was explained in \cite[Section 6.1]{Pos11}, alternatively the arguments in the proof of Lemma \ref{lem:semicomplete} apply with obvious modifications. 
The adjointness follows as in the non-curved case; namely by noticing that for $A\in \CDG$, $C\in\pcCDG$ the sets of morphisms $\Hom_{\CDG}(\Omega C,A)$ and $\Hom_{\pcCDG}(\widecheck{B}A,C)$ are both naturally isomorphic to $\MC(A\otimes C)$.
\end{proof}

\begin{rem}
If a curved dg algebra $A$ is happens to be augmented, then there is a natural choice of a retraction $\epsilon\colon A\to k$, namely, the given augmentation. 
In this case, $\widecheck{B}A$ is uncurved. 
Similarly, if $A$ has vanishing curvature, $\widecheck{B}A$ is naturally augmented. 
If $A$ is both augmented and uncurved, then so is $\widecheck{B}A$.
\end{rem}

Now for a curved dg algebra $A$ and its bar construction $\widecheck{B}A$, there is an adjunction
\begin{equation}\label{eq:adj}
G \colon \pcDGMod \widecheck{B}A^{\op} \rightleftarrows \DGMod A \lon F
\end{equation}
as defined in \cref{functorsFG}; these functors are well-defined as the twisting of a curved dg algebra by a Maurer--Cartan element gives an uncurved dg algebra. 
Furthermore, Theorem \ref{thm:modelcat} holds (with the same definitions of weak equivalences, fibrations and cofibrations) when the dg coalgebra $C$ is curved (indeed, this is how it was formulated in \cite{Pos11}). 
Thus, $\pcDGMod \widecheck{B}A^{\op}$ has the structure of a model category and by transferring along the adjunction (\ref{eq:adj}) we obtain the following generalisation of \cref{modcat}; the arguments are the same as in the uncurved case. 

\begin{thm}
Let $A$ be a curved dg algebra.
There is a cofibrantly generated model category structure on\/ $\DGMod A$, where 
\begin{enumerate}
\item a morphism $f \colon M \to N$ is a \emph{weak equivalence} if it induces a quasi-isomorphism
\[
\Hom_A((V \otimes A)^{[x]}, M) \to \Hom_A((V \otimes A)^{[x]}, N)
\]
for any finitely generated twisted $A$-module $(V \otimes A)^{[x]}$;
\item a morphism is a \emph{fibration} if it is surjective;
\item a morphism is a \emph{cofibration} if it has the left lifting property with respect to acyclic fibrations.
\end{enumerate}
With this model structure, the adjunction\/ $G \dashv F$ is a Quillen pair.
\end{thm}

Similarly, there are model structures on $\DGMod A$ when $A$ is curved and augmented, or non-curved and non-augmented. 
Altogether there are four cases as below. 
Case (\ref{prevcase}) is the case considered previously and proved in \cref{Qequiv}. Again, the arguments employed in the augmented uncurved case generalise in a straightforward fashion.

\begin{thm}
With the above model structures, the functors\/ $G \dashv F$ form a Quillen anti-equivalence between the categories\/ $\pcDGMod \widecheck{B}A$ and\/ $\DGMod A$ in each of the following four cases:
\begin{enumerate}
\item $A$ is curved and non-augmented, $\widecheck{B}A$ is curved and non-augmented; 
\item $A$ is curved and augmented, $\widecheck{B}A$ is non-curved and non-augmented; 
\item $A$ is non-curved and non-augmented, $\widecheck{B}A$ is curved and augmented; 
\item\label{prevcase} $A$ is non-curved and augmented, $\widecheck{B}A$ is non-curved and augmented. 
\end{enumerate}
\end{thm}

%\bibliographystyle{abbrvurl}
%\bibliography{references}

\setlength{\parindent}{0pt}
\end{document}